\documentclass[12pt, reqno]{amsart}
\usepackage{amssymb, amsthm, amsmath, amsfonts}
\usepackage{array, epsfig}
\usepackage{bbm}

  \evensidemargin
\oddsidemargin
\begin{document}

\newcommand{\be}{\begin{equation}}
\newcommand{\ee}{\end{equation}}
\newcommand{\bea}{\begin{eqnarray}}
\newcommand{\eea}{\end{eqnarray}}
\newcommand{\beaa}{\begin{eqnarray*}}
\newcommand{\eeaa}{\end{eqnarray*}}

\renewcommand{\proofname}{\bf Proof}
\newtheorem*{rem*}{Remark}
\newtheorem*{cor*}{Corollary}
\newtheorem{cor}{Corollary}
\newtheorem{prop}{Proposition}
\newtheorem*{prop1}{Proposition 1}
\newtheorem{lem}{Lemma}
\newtheorem*{lem1'}{Lemma $\mathbf{1^\prime}$}
\newtheorem{theo}{Theorem}
\newfont{\zapf}{pzcmi}

\def\R{\mathbb{R}}
\def\Z{\mathbb{Z}}
\def\N{\mathbb{N}}
\def\E{\mathbb{E}}
\def\P{\mathbb{P}}
\def\V{\mathbb{D}}
\def\I{\mathbbm{1}}
\newcommand{\D}{\hbox{\zapf D}}
\newcommand{\Pt}{\widetilde{\mathbb{P}}}
\newcommand{\Et}{\widetilde{\mathbb{E}}}

\title{Positivity of integrated random walks}
\author[V. Vysotsky]{Vladislav Vysotsky}
\address{Arizona State University, St.Petersburg Department of Steklov Mathematical Institute, and Chebyshev Laboratory at St.Petersburg State University}
\thanks{This research is supported by the Chebyshev Laboratory (Department of Mathematics and Mechanics, St.-Petersburg State University) under RF government grant 11.G34.31.0026 and by the grant 10-01-00242 of RFBR}
\email{vysotsky@asu.edu, vysotsky@pdmi.ras.ru}
\subjclass[2000]{60G50, 60F99}
\keywords{Integrated random walk, persistence, one-sided exit probability, unilateral small deviations, area of random walk, Sparre-Andersen theorem, stable excursion, area of excursion}

\begin{abstract}
Take a centered random walk $S_n$ and consider the sequence of its partial sums $A_n:=\sum_{i=1}^n S_i$. Suppose $S_1$ is in the domain of normal attraction of an $\alpha$-stable law with $1 < \alpha \le 2$. Assuming that $S_1$ is either right-exponential (that is $\P(S_1>x|S_1>0)=e^{-ax}$ for some $a>0$ and all $x>0$) or right-continuous (skip free), we prove that $$\P \Bigl \{ A_1 > 0, \dots, A_N > 0 \Bigr \} \sim C_\alpha N^{\frac{1}{2\alpha} - \frac12}$$ as $N \to \infty$, where $C_\alpha >0$ depends on the distribution of the walk. We also consider a conditional version of this problem and study positivity of integrated discrete bridges.
\end{abstract}

\maketitle

\section{Introduction}
\subsection{The problem}
Consider a non-degenerate sequence of centered random variables. What is the probability that it stays positive for a long time? Surprisingly little is known about this problem. Only one situation is well understood besides the trivial case that the variables are independent: For a random walk $S_n$, the classical Sparre-Andersen theorem expresses the generating function of $$q_N:=\P \Bigl \{ \min \limits_{1 \le k \le N} S_k  > 0 \Bigr \}$$ in terms of the probabilities $\P(S_n >0)$. A Tauberian theorem then implies that $N^{1/2} q_N \to c >0$ in the typical case that $\E S_1 =0, Var(S_1) < \infty$; moreover, if $\P(S_n >0) \to \gamma \in (0,1)$, then $N^{1-\gamma} q_N$ is slowly varying at infinity.

Consider the sequence $A_n:=\sum_{i=1}^n S_i$, which we call an {\it integrated random walk}. We are interested in the asymptotics of $$p_N:=\P \Bigl \{ \min \limits_{1 \le k \le N} A_k  > 0 \Bigr \}$$ as $N \to \infty$. One may also refer to similar types of questions as to asymptotics of the tail of one-sided exit times, unilateral small deviation probabilities, or {\it persistence} if adopting the terminology from physics.

This problem was introduced in the seminal paper by Sinai~\cite{Sinai} who considered the specific case that $S_n$ is a simple random walk. Sinai studied the question in connection with the behavior of solutions of the Burgers equation with random initial data. The author's initial motivation comes from his study~\cite{ISticky} of sticky particle systems with gravitational attraction. The asymptotical behavior of $p_N$ is directly related to the characteristics of such systems with random initial data at the critical moment of total gravitational collapse. The probabilities $p_N$ also arise in the wetting model of random polymers with Laplacian interaction considered by Caravenna and Deuschel~\cite{Polymers}. More generally, the probability that a certain random function does not change sign over a large time scale is relevant to the analysis of many physical models, Majumdar~\cite{Majum}.

Although continuous-time versions of our question have received more attention, there are few results even in this direction. Aurzada and Dereich~\cite{Germans} give a comprehensive overview of this work.  A recent breakthrough \cite{Germans} shows universality of the asymptotics in the {\it one-sided exit problem} for general integrated L\'{e}vy processes.

\subsection{The background}
The first result on the subject is due to Sinai~\cite{Sinai} who explained that $p_N \asymp N^{-1/4}$ for an (integrated) simple random walk. Because the continuous-time analog with an integrated Wiener process $A(t):=\int_0^t W(s) ds$ exhibits the same asymptotics and moreover, $$\P \{ \inf \limits_{0 \le t \le N} A(t) \ge -1 \} \sim c N^{-1/4}$$ (see, for example, Isozaki and Watanabe~\cite{Japan}), it was conjectured in \cite{Polymers, ISticky} that $p_N \asymp N^{-1/4}$ for any walk $S_n$ with $\E S_1 =0$ and $Var(S_1) < \infty$. This conjecture has not yet been fully proved and below we briefly explain existing approaches.

Sinai's method relies on the observation that if $S_n$ is a simple random walk, then all the local extrema of $A_n$ occur at the times when $S_n$ returns to zero, and such times form a renewal sequence. This property is based on the very specific structure of the increments of the walk and does not hold for different distributions. However, the main message here is to partition the trajectory of $S_n$ with a suitable sequence of regeneration times. Vysotsky~\cite{IIRW} explored this idea and showed that $p_N \lesssim N^{-1/4}$ for {\it any integer-valued} walk (we write $a_n \lesssim b_n$ for two non-negative sequences $a_n$ and $b_n$ if $a_n/b_n$ stays bounded while $a_n \asymp b_n$ means $a_n \lesssim b_n$ and $b_n \lesssim a_n$). Unfortunately, further development of this method required restrictive assumptions on the positive increments of the walk. Due to technical difficulties,~\cite{IIRW} also imposed similar constraints on the negative increments and showed that $p_N \asymp N^{-1/4}$ for double-sided exponentials, symmetric geometric, lazy simple, and two other ``mixed'' types of random walks.  We stress that the present paper removes these superfluous assumptions on the negative increments.

The second approach is due to Aurzada and Dereich~\cite{Germans} who used strong approximation by a Wiener process assuming $\E e^{a |S_1|} < \infty$ for some $a>0$. This powerful method allowed them to prove universality of the asymptotics for general integrated Levy processes but, because the strong approximation technique does not work well for small values of time,~\cite{Germans} obtains extra factors in the estimates: $N^{-1/4} (\log N)^{-4} \lesssim p_N \lesssim N^{-1/4} (\log N)^4$.

The third method by Dembo and Gao~\cite{Dembo} is based on decomposition of the sequence $A_n$ at its maximum. \cite{Dembo} proved that $p_{N-1} \le c_1 (\E|S_N|/N)^{1/2}$ for {\it any} (centered) walk; this bound clearly scales as the desired $N^{-1/4}$ for walks with $Var(S_1) < \infty$. For the sharp lower bound, Dembo and Gao still had to impose assumptions on the positive increments of the walk. They showed that $c_2 (\E|S_N|/N)^{1/2} \le p_{N-1}$ assuming, essentially, that the tail $\P(S_1 > x)$ decays exponentially or super-exponentially. Thus the results of \cite{Dembo} cover those of \cite{IIRW} and essentially of \cite{Germans}.

\subsection{Results and organization of the paper}
This paper proves the sharp asymptotics for $p_N$ in certain cases. We follow the approach developed in \cite{IIRW} but work under much less restrictive conditions on the increments of $S_n$. Presently, it seems impossible to get the sharp asymptotics using the other methods described above.

Let us state the assumptions. A random walk $S_n$ is {\it right-exponential} if $\mbox{Law} (S_1 | S_1>0)$ is an exponential distribution. An integer-valued walk $S_n$ is {\it right-continuous} (skip free) if $\P\{ S_1=1|S_1>0\} = 1$; the name comes from the analogy with spectrally negative integrable L\'{e}vy processes, which do not have positive jumps and hit all intermediate values before reaching any positive horizontal level. The distributions above are well known in the renewal theory and have the characteristic property that all overshoots of $S_n$ over any fixed level are identically distributed.

Suppose that $S_1$ belongs to the domain of normal attraction (to be denoted as $S_1 \in \mathcal{DN}(\alpha)$) of a strictly stable law with the index $1 < \alpha \le 2$. If $\alpha < 2$ and $S_1$ is right-exponential or right-continuous, then such a law is spectrally negative. By the stable central limit theorem and Eq. (2.2.30) in Zolotarev~\cite{Zol} for the positivity parameter, it holds that $\P\{S_n > 0\} \to 1/\alpha$. For $1 < \alpha \le 2$, define
\beaa
\mathcal{R}_\alpha &:=& \Bigl \{ S_1: S_n \mbox{ is either right-exponential or right-continuous}, \E S_1 = 0, \\
&& S_1 \in \mathcal{DN}(\alpha) \mbox{, and } \sum_{n=1}^\infty \frac1n \Bigl(\P\{S_n > 0\}  - \frac1\alpha \Bigr) \mbox { converges} \Bigr \}.\\
\eeaa
Recall that (Theorem 2.6.6 in Ibragimov and Linnik~\cite{IbrLin}) $S_1 \in \mathcal{DN}(2)$ is equivalent to $Var(S_1) < \infty$, which ensures the convergence of the series (Feller~\cite[Ch. XVIII.5]{Feller}). Due to Egorov~\cite{Egorov}, for $1< \alpha <2$ a sufficient condition for the convergence is $\int_0^{\infty} x^\alpha |d(F(-x)-G_\alpha(-x)| < \infty$, where $F(x)$ and $G_\alpha(x)$ are the distribution functions of $S_1$ and the limit stable law, respectively.

We now state the main result of the paper.

\begin{theo} \label{SHARP}
Let $S_n$ be a random walk such that $S_1 \in \mathcal{R}_\alpha$ for some $1 < \alpha \le 2$. Then there exists a constant $C_\alpha=C_\alpha(Law(S_1))>0$ such that $$\lim_{N \to \infty} N^{\frac12 - \frac{1}{2\alpha}} p_N = C_\alpha.$$
\end{theo}
\begin{rem*}
The computable bounds for $C_2$ when $S_1$ is upper-exponential are given below in \eqref{c}.
\end{rem*}
Our proof should also work if we drop convergence of the series in the definition of $\mathcal{R}_\alpha$. Then $N^{\frac12 - \frac{1}{2\alpha}} p_N$ becomes  slowly varying at infinity instead of being convergent. We also point out that \cite{Dembo} proved the weak asymptotics $p_N \asymp N^{\frac{1}{2\alpha} - \frac12}$ for centered random walks with $S_1 \in \mathcal{DN}(\alpha)$ whose tail $\P(S_1 > x)$ decays exponentially or super-exponentially. This class of distributions is much wider than $\mathcal{R}_\alpha$.

We can also apply our method to a conditional version of the problem with integrated discrete bridges instead of integrated random walks. For an integer-valued walk $S_n$, put $$p^*_N:=\P \Bigl \{ \min \limits_{1 \le k \le N} A_k  > 0 \bigl | S_N = 0 \bigr. \Bigr \}$$ for $N \in \mathcal{D}_{S_1} := \{n: \P(S_n = 0) >0 \}$ where this expression is well-defined.

\begin{prop} \label{BRIDGE}
Let $S_n$ be an integer-valued random walk with $\E S_1=0$ and $Var(S_1) < \infty$. Then $p^*_N \lesssim N^{-1/4}$ and moreover, $p^*_N \asymp N^{-1/4}$ if $S_1$ is right-continuous, as $N \to \infty$ along $\mathcal{D}_{S_1}$.
\end{prop}

Although this statement covers quite a narrow class of distributions, it is the first result of such type. An open and very challenging problem that has recently received some attention is to find the asymptotics of
$$\P \Bigl \{ \min \limits_{1 \le k \le N} A_k  > 0 \bigl | A_N = 0 \bigr. \Bigr \} \quad \mbox{and} \quad \P \Bigl \{ \min \limits_{1 \le k \le N} A_k  > 0 \bigl | S_N =0, A_N = 0 \bigr. \Bigr \}.$$ These probabilities are related to polymer models similar to the one of Caravenna and Deuschel~\cite{Polymers}.

This paper is organized as follows. Sec.~\ref{SEC NON SHARP} explains our approach of partitioning the trajectory of $S_n$ into independent parts (so-called cycles) by the appropriate moments of regeneration. The pivotal result of the section is Proposition~\ref{INDEP} on bivariate random walks that stay in the right half-plane. Roughly speaking, it is a bivariate version of the famous Sparre-Andersen theorem that $q_N$ does not depend on the distribution of the walk if $S_1$ is symmetric and continuous. In addition to its independent interest, Proposition~\ref{INDEP} leads to a simple and very intuitive proof that $p_N \asymp N^{\frac{1}{2\alpha} - \frac12}$ for $S_1 \in \mathcal{R}_\alpha$; one may regard this proof as a rigorous version of the heuristic arguments in \cite[Sec. 2.1]{IIRW}. We give the proof here to make the paper more readable and to show the advantage of our technique. In Sec.~\ref{SSEC BRIDGE} we apply Proposition~\ref{INDEP} to get our results on the positivity of integrated bridges. Theorem~\ref{SHARP} is proved in Sec.~\ref{SEC SHARP}. The necessary ingredients are prepared in Sec.~\ref{SEC EXCURSIONS}, where we study joint tails of areas and lengths of stable excursions, cycles and meanders. There we also discuss conditional limit theorems for bivariate random walks, one of the main tools in Sec.~\ref{SEC SHARP}.

\section{Partitioning into cycles and non-sharp asymptotics of $p_N$} \label{SEC NON SHARP}

\subsection{Partitioning by regenerating times} \label{PARTITION}
The main idea of our approach is to partition the trajectory of the random walk $S_n$ into appropriate independent parts. Define the moments of {\it crossing} the zero level {\it from below} as $$\Theta_0:=\min\{n \ge 0: S_{n+1} > 0\}, \quad \Theta_{k+1}:= \min\{n > \Theta_k: S_n \le 0, S_{n+1} > 0\}$$ for $k \ge 0$. A non-degenerate centered random walk is recurrent hence the r.v.'s defined above are proper.  We stress that although the variables $\Theta_k + 1$ are stopping times, the variables $\Theta_k$ are not. The trajectory of $S_n$ is thus partitioned into parts that we call {\it cycles} (except for the part until $\Theta_0$ that will be excluded from the consideration), and each cycle starts with a positive excursion  followed by a negative excursion. For $k \ge 1$, let $\theta_k:=\Theta_k-\Theta_{k-1}$ be the length of the $k$'th cycle and let $\psi_k:=A_{\Theta_k}-A_{\Theta_{k-1}}$ be its area; also, set $\Psi_k:=A_{\Theta_k}$ for the total area of the first $k$ cycles so that $\psi_k=\Psi_k -\Psi_{k-1}$.

Define $\Pt(\cdot):=\P(\cdot | S_1>0)$ as it is more convenient to assume that $S_n$ starts with a positive excursion and so $\Theta_0=0$ $\Pt$-a.s.; also put $\sigma^2:=Var(S_1)$. The following  observation from Vysotsky~\cite{IIRW} (see Lemmae 1, 2 and Proposition 1) plays the crucial role for our method.

\begin{lem} \label{IID}
Let $S_n$ be a centered random walk that is either right-exponential or right-continuous. Then $(\theta_n, \psi_n)_{n \ge 1}$ are i.i.d. and $(\theta_1, \psi_1) \stackrel{\D}{=} (\theta_1, -\psi_1)$. If $S_1 \in \mathcal{R}_2$, then $\theta_1 \in \mathcal{DN}(1/2)$, and, moreover, $\lim \limits_{n \to \infty} n^{1/2} \P \{ \theta_1 \ge n \} = \sqrt{\frac{8}{\pi}} \frac{\sigma}{\E |S_1|}$ if $S_1$ is right-exponential.
\end{lem}

Here is an explanation of this result. The i.i.d. property follows as each cycle starts with the overshoot $S_{\Theta_k+1}$ which is independent of the preceding part $S_1, \dots , S_{\Theta_k}$ of the trajectory. The symmetry holds by
\be \label{symm}
\bigl(S_1, \dots , S_{\hat{\theta}_1}, \hat{\theta}_1 \bigr)  \stackrel{\D}{=} \bigl(-S_{\hat{\theta}_1}, \dots, -S_1, \hat{\theta}_1 \bigr) \quad \mbox{under } \Pt,
\ee
where $\hat{\theta}_1:=\max\{n \le \theta_1: S_{\Theta_0 + n} <0 \}$. Of course $\hat{\theta}_1=\theta_1$ for a right-exponential $S_1$ while for a right-continuous $S_1$, $\hat{\theta}_1$ equals the length of the first cycle by the last return to zero. The key observation is the  duality relation
\beaa
&& \E\{ S_2 \in d x_2, \dots, S_{k-1} \in d x_{k-1} | S_1=x_1, S_k=x_k\} \\
&=& \E\{ S_{k-1} \in -d x_2, \dots, S_2 \in -d x_{k-1} | S_1=-x_k, S_k=-x_1\},
\eeaa
which holds for any random walk and follows from the standard duality principle that states $(S_2 -S_1, \dots, S_k-S_1) \stackrel{\D}{=} (S_k - S_{k-1}, \dots, S_k-S_1)$. The duality immediately implies \eqref{symm} if $S_1$ is right-continuous, while for the right-exponential case we use the fact that
\beaa
&& \Pt \{ S_1 \in d x_1, \dots, S_k \in d x_k, \hat{\theta}_1 = k\} \\
&=& \P(S_1>0) e^{x_k-x_1} \E\{ S_2 \in d x_2, \dots, S_{k-1} \in d x_{k-1} | S_1=x_1, S_k=x_k\} d x_1 \P \{ S_k - S_1 \in d x_k - x_1 \}.
\eeaa

Relation \eqref{symm} is extremely useful as it essentially states that the negative part of a cycle has the same distribution as the time-reversed positive part. Hence for a right-exponential $S_1$ we immediately get $\theta_+ \stackrel{\D}{=} \theta_-$ for the lengths of the parts that are defined as $$\theta_+:= \min \{n \ge 1: S_{\Theta_0 + n} \ge 0,  S_{\Theta_0 + n+ 1} <0 \}, \quad \theta_-:= \theta_1 - \theta_+.$$ To cover the right-continuous case, define
$$\hat{\theta}_+:=\max\{n \le \theta_+: S_{\Theta_0 + n} >0 \}, \quad \hat{\theta}_-:=\hat{\theta}_1 - \theta_+$$ to be the lengths of the parts by their last returns to zero. Then we have $\hat{\theta}_+  \stackrel{\D}{=} \hat{\theta}_-$, which clearly covers the first case as $\hat{\theta}_+ = \theta_+$ and $\hat{\theta}_- = \theta_-$ a.s. for a right-exponential $S_1$.

Since the Sparre-Andersen theorem implies that $\theta_+ \in \mathcal{DN}(1-1/\alpha)$, by $\hat{\theta}_+  \stackrel{\D}{=} \hat{\theta}_-$ the same holds for $\theta_-$. Hence one also expects $\theta_1 \in \mathcal{DN}(1-1/\alpha)$ as stated in Lemma~\ref{IID} for $\alpha=2$. We prove this for $1 < \alpha < 2$ in the next section. Note that there is a significant difference in the shape of long cycles: for $\alpha = 2$ the walk essentially stays either positive or negative while for $1 < \alpha <2$ it spends positive parts of time in both half-planes.

\subsection{Bivariate walks staying in the right half-plane}
We see that  under the conditions of Lemma~\ref{IID}, $(\Theta_k, \Psi_k)$ is a bivariate (two-dimensional) random walk. Its second component is symmetric, and the walk starts at $(\Theta_0, \Psi_0) = (0, 0)$ under $\Pt$. It turns out that such walks enjoy a useful property stated below in Proposition~\ref{INDEP}. This result actually is a slight improvement of the famous Sparre-Andersen theorem which states, in particular, that the probabilities $q_N$ that a symmetric continuous random walk stays positive until time $N$ do not depend on the distribution of the increments. Proposition~\ref{INDEP} is inspired by Lemma~3 (see \eqref{Sinai L} below) by Sinai~\cite{Sinai}; we also refer to Feller~\cite[Ch. XII]{Feller} for appropriate definitions and general ideas. It is worth mentioning that Sinai's lemma is a special case of the results by Greenwood and Shaked~\cite{Greenwood} who give a half-plane Wiener-Hopf type factorization for bivariate distributions. This reference was pointed out to the author by Vitaliy Wachtel.

\begin{prop}\label{INDEP}
Let $(S_n^{(1)}, S_n^{(2)})$ be a bivariate random walk such that $(S_1^{(1)}, S_1^{(2)}) \stackrel{\D}{=} (S_1^{(1)}, -S_1^{(2)})$. Then for any $n \ge 1$ and $x \in \R$ it holds
\be \label{indep1}
\P \Bigl \{ S_n^{(1)} \in dx, \min_{1 \le i \le n} S_i^{(2)} \ge 0 \Bigr\} \ge \P \Bigr\{  S_n^{(1)} \in dx \Bigr\} \P \Bigl \{ \min_{1 \le i \le n} S_i^{(2)} > 0 \Bigl \}
\ee
and
\be \label{indep2}
\P \Bigl \{ S_n^{(1)} \in dx, \min_{1 \le i \le n} S_i^{(2)} > 0 \Bigr\} \le \P \Bigr\{  S_n^{(1)} \in dx \Bigr\} \P \Bigl \{ \min_{1 \le i \le n} S_i^{(2)} \ge 0 \Bigl \}.
\ee
\end{prop}
\begin{cor*}
If the distribution of $S_1^{(2)}$ is continuous, then $\bigl \{ S_n^{(1)} \in dx \bigr\}$ and $\bigl \{ \min \limits_{1 \le i \le n} S_i^{(2)} > 0 \bigr\}$ are independent.
\end{cor*}

\begin{proof}
It is useful to consider the inequalities simultaneously for all $n \ge 1$ and think that the sides of \eqref{indep1} define the two measures $Q(dx, n)$ and $\tilde{Q}(dx, n)$ on $\R \times \N$. We start by noting that the characteristic function $$\chi(u,t):=\sum_{n=1}^\infty \int_{-\infty}^\infty t^n e^{iux} \P \Bigl \{ S_n^{(1)} \in dx, \min_{1 \le i \le n} S_i^{(2)} \ge 0 \Bigr\} $$ of the measure $Q(dx, n)$ in the left-hand sides of \eqref{indep1} satisfies
$$1 + \chi(u,t) = \frac{1}{1-\chi_{U_1, T_1} (u, t)},$$
where $T_1$ is the first weak ascending ladder epoch of the walk $S^{(2)}$ and $U_1 := S_{T_1}^{(1)}$. This follows by the standard argument of considering the equally distributed dual walk $(S_n^{(1)} - S_{n-k}^{(1)}, S_n^{(2)} - S_{n-k}^{(2)})_{ 1 \le k \le n}$ and getting
\bea
\P \Bigl \{ S_n^{(1)} \in dx, \min_{1 \le i \le n} S_i^{(2)} \ge 0\Bigr\} &=& \P \Bigl \{  S_n^{(1)} \in dx, \max_{1 \le i \le n-1} S_i^{(2)} \le S_n^{(2)} \Bigr\} \notag\\
&=& \sum_{k=1}^\infty \P \Bigl \{  U_1 + \dots + U_k \in dx, T_1 + \dots + T_k =n \Bigr\}, \label{expln}
\eea
where $(U_k, T_k)_{k \ge 1}$ are i.i.d. random vectors.

Further, Lemma 3 by Sinai~\cite{Sinai}, which is just a two-dimensional version of the Sparre-Andersen theorem (see Feller~\cite[Sec. XII.7, Theorem 1]{Feller}), states
\be \label{Sinai L}
\log \frac{1}{1-\chi_{U_1, T_1}(u,t)} = \sum_{n=1}^\infty \int_{-\infty}^\infty \frac{t^n}{n} e^{iux} \P \Bigl \{ S_n^{(1)} \in dx, S_n^{(2)} \ge 0 \Bigr\},
\ee
hence with the symmetry of $S_n^{(2)}$,
$$\log \frac{1}{1-\chi_{U_1, T_1}(u,t)} = \sum_{n=1}^\infty \frac{t^n}{2n} \chi_{S_1^{(1)}}^n(u) + \sum_{n=1}^\infty \int_{-\infty}^\infty \frac{t^n}{2n} e^{iux} \P \Bigl \{ S_n^{(1)} \in dx, S_n^{(2)} = 0 \Bigr\}.$$
The second term in the right-hand side can be transformed we did in \eqref{expln} and we get
$$
\log \frac{1}{1-\chi_{U_1, T_1}(u,t)} = \frac12 \log \frac{1}{ 1 - t\chi_{S_1^{(1)}}(u)} + \frac12 \log \frac{1}{1-\chi_{U_1^*, T_1^*}(u,t)},
$$
where $\chi_{U_1^*, T_1^*}(u,t)$ is the characteristic function of the non-probability measure $$Q^*(dx, k):= \P \bigl\{ S_k^{(1)} \in dx , S_1^{(2)}<0, \dots, S_{k-1}^{(2)}<0, S_k^{(2)}=0\bigr\}$$ on $\R \times \N$ which in a certain sense is the distribution of the defective random vector $(U_1^*, T_1^*)$, where $T_1^*$ is the first moment when $S_n^{(2)}$ hits zero from below and $U_1^*=S_{T_1^*}^{(1)}$. Then
\be \label{cfLHS}
1 + \chi(u,t) = \sqrt{ \frac{1}{\bigl(1 - t\chi_{S_1^{(1)}}(u)\bigr)\bigl(1-\chi_{U_1^*, T_1^*}(u,t)\bigr)}},
\ee
and similarly, the characteristic function $\chi^+(u,t)$ of the measure $Q^+(dx,n)$ in the left-hand sides of \eqref{indep2} satisfies
\be \label{cfLHS2}
1 + \chi^+(u,t) = \sqrt{ \frac{1-\chi_{U_1^*, T_1^*}(u,t)}{1 - t\chi_{S_1^{(1)}}(u)}}.
\ee

Finally, the characteristic function of the measure $\tilde{Q}(dx,n)$ in the right-hand sides of \eqref{indep1} is
$$\tilde{\chi}(u,t) = \sum_{n=1}^\infty  t^n \chi_{S_1^{(1)}}^n(u) \P \Bigl \{ \min_{1 \le i \le n} S_i^{(2)} > 0 \Bigl \} =\chi_{T_1^+} \bigl(t \chi_{S_1^{(1)}}(u) \bigr),$$
where $T_1^+$ is the first strong ascending ladder epoch of $S^{(2)}$, and by \eqref{cfLHS}, \eqref{cfLHS2} and $\chi_{T_1^+} (t) = \chi^+ (0, t)$, we get
\be \label{cg=cg}
1 + \chi(u,t) = \bigl(1 + \tilde{\chi}(u,t) \bigr) \sqrt{ \frac{1}{ \bigl( 1-\chi_{T_1^*}(t \chi_{S_1^{(1)}}(u)) \bigr ) \bigl( 1-\chi_{U_1^*, T_1^*}(u,t) \bigr)}}.
\ee
Since all the coefficients of the Maclaurin series of $(1-z)^{-1/2}$ are positive, the square root factor in the right-hand side of \eqref{cg=cg} has the form
$$1 + \sum_{k, m \ge 1} a_{k, m} \chi^k_{T_1^*}(t \chi_{S_1^{(1)}}(u)) \chi^m_{U_1^*, T_1^*}(u,t)=: 1 + \phi(u,t)$$ with some $a_{k,m} >0$. As products of characteristic functions correspond to convolutions of measures, $\phi(u,t)$ is the characteristic function of some {\it measure} $D(dx,n)$ on $\R \times \N$. Then $$Q(dx, n) = \tilde{Q}(dx,n)+ D(dx, n) + \tilde{Q}(dx,n)*D(dx, n) \ge \tilde{Q}(dx,n)$$ implying \eqref{indep1}. A similar argument concludes \eqref{indep2}.
\end{proof}

Note that \eqref{cfLHS} and \eqref{cfLHS2} actually follow from Eq. (4) by Greenwood and Shaked~\cite{Greenwood} with $\tau$ and $\nu$ from their Example (a) on p. 568, but we have chosen to start from the Sinai lemma to go along the lines of our original proof.

\subsection{Weak asymptotics of $p_N$}
The next statement is not new since it follows from Theorem 1.2 by Dembo and Gao~\cite{Dembo}. As explained in the introduction, we give the proof here to show a simple and very intuitive way to understand the asymptotics of $p_N$ and demonstrate advantage of our technique.

\begin{prop} \label{WEAK}
If $S_1 \in \mathcal{R}_\alpha$ for some $1 < \alpha \le 2$, then $ p_N \asymp N^{\frac{1}{2\alpha} - \frac12}.$
\end{prop}
\begin{rem*}
If $S_1$ is right-exponential and $S_1 \in \mathcal{R}_2$, then
\be \label{c}
\Bigl [ \varliminf_{N \to \infty} p_N  N^{\frac12 - \frac{1}{2\alpha}} , \varlimsup_{N \to \infty} p_N  N^{\frac12 - \frac{1}{2\alpha}}\Bigr] \subset \frac{2^{1/4} }{\pi} \Gamma{\Bigl(\frac14 \Bigr)} \sqrt{\frac{\sigma}{\E|S_1|}}  \P\{ S_1 >0\} \times \Bigl[\frac12, 1 \Bigr].
\ee
\end{rem*}

\begin{proof}
The key observation is that
\be \label{reduction}
\P \Bigl \{ \min \limits_{1 \le k \le N} A_k  > 0 \Bigr \} = \P \Bigl \{ \min_{1 \le k \le \eta(N)} A_{\Theta_k}  > 0 , \, A_1 >0 , A_N >0  \Bigr \},
\ee
where $$\eta(N):= \max \{n \ge 0: \Theta_n \le N \}.$$ Under $\Pt$, this quantity is just the number of up-crossing of the zero level by the walk $S_n$ by the time $N$. Then
$$ \Pt \Bigl \{ \min \limits_{1 \le k \le \eta(N)+1} \Psi_k  > 0 \Bigr \} \le \frac{p_N}{\P\{ S_1 >0 \}} \le \Pt \Bigl \{ \min \limits_{1 \le k \le \eta(N)} \Psi_k  > 0 \Bigr \}.$$ For the lower bound, our idea is to flip the last incomplete cycle using the conditional symmetry of $\psi_i$ (Lemma~\ref{IID}) to make sure it has a positive area: condition on $\eta(N)$ and $\Theta_{\eta(N)}$ and get
\bea \label{condit}
&& \Pt \Bigl \{ \min \limits_{1 \le k \le \eta(N)} \Psi_k  > 0, \psi_{\eta(N)+1} \ge 0 \Bigr \} \notag \\
&=& \sum_n \sum_{i \le N} \Pt \Bigl \{\Theta_n = i, \theta_{n+1} > N-i, \min \limits_{1 \le k \le n} \Psi_k  > 0, \psi_{n+1} \ge 0 \Bigr \} \\
&=& \sum_n \sum_{i \le N} \Pt \Bigl \{\Theta_n = i, \min \limits_{1 \le k \le n} \Psi_k  > 0\Bigr \} \Pt \Bigl \{\theta_{n+1} > N-i, \psi_{n+1} \ge 0 \Bigr \} \notag\\
&\ge& \frac12 \Pt \Bigl \{ \min \limits_{1 \le k \le \eta(N)} \Psi_k  > 0 \Bigr \}. \notag
\eea
Hence
\be \label{disregard}
\frac12 \Pt \Bigl \{ \min \limits_{1 \le k \le \eta(N)} \Psi_k  > 0 \Bigr \} \le \frac{p_N}{\P\{ S_1 >0 \}} \le \Pt \Bigl \{ \min \limits_{1 \le k \le \eta(N)} \Psi_k  > 0 \Bigr \}.
\ee

{\it Case 1: $S_1$ is right-exponential.} By conditioning on $\eta(N)$ and using Proposition~\ref{INDEP}, we proceed as above in \eqref{condit} and get the most important relation
\be \label{key}
\Pt \Bigl \{ \min \limits_{1 \le k \le \eta(N)} \Psi_k  > 0 \Bigr \} = \sum_{n = 0}^\infty \Pt \bigl \{\eta(N) = n \bigr \} \Pt \Bigl \{ \min \limits_{1 \le k \le n} \Psi_k  > 0 \Bigr \}.
\ee
As the distribution of $\Phi_k$ is symmetric, the Sparre-Andersen theorem implies the existence of a positive limit $$c_1:=\lim \limits_{n \to \infty} n^{1/2} \Pt \Bigl \{ \min \limits_{1 \le k \le n} \Psi_k  > 0 \Bigr \},$$ hence
\bea
\Pt \Bigl \{ \min \limits_{1 \le k \le \eta(N)} \Psi_k  > 0 \Bigr \} &=& \sum_{k = 0}^\infty \Pt \Bigl \{\eta(N) = n \Bigr \} \frac{c_1 + o(1)}{\sqrt{n+1}} \notag \\
&=& (c_1 + o(1)) \Et \frac{1}{\sqrt{\eta(N)+1}} + O \Bigl(\Pt \bigl \{\eta(N) < \ln N \bigr \} \Bigr) \notag\\
&=& \frac{c_1 + o(1)}{N^{\frac12 - \frac{1}{2\alpha}}} \Et \sqrt{\frac{N^{1-1/\alpha}}{\eta(N)+1}}, \label{keyk}
\eea
where we used that
$$\Pt \bigl \{\eta(N) < \ln N \bigr \}  = \Pt \bigl \{\Theta_{\ln N} > N \bigr \} \le \ln N \, \Pt \bigl \{\theta_1 > N / \ln N \bigr \} = o(N^{\frac{1}{2\alpha}- \frac12}).$$

It remains to check that the expectations stay bounded away from zero and infinity. Let $\Theta_k^+$ and $\Theta_k^-$ be the total length of the first $k$ positive and $k$ negative excursions of $S_n$, respectively. Under $\Pt$ it is true that $\Theta_k^+$ and $\Theta_k^-$ are random walks with the increments distributed as $\theta_+$ and $\theta_-$, respectively, and $\Theta_k= \Theta_k^+ + \Theta_k^-$. Define the numbers of renewal epochs $\eta^+(N)$ and $\eta^-(N)$ analogously to $\eta(N)$, then $\min(\eta^+(N/2), \eta^-(N/2)) \le \eta(N) \le \eta^+(N)$ and it suffices to consider the expectations in \eqref{keyk} with $\eta^+(N)$ and $\eta^-(N)$ instead of $\eta(N)$.

As $\theta_+, \theta_- \in \mathcal{DN}(1-1/\alpha)$, it follows from Feller~\cite[Ch. XI.5]{Feller} that $\eta^+(N)$ and $\eta^-(N)$ satisfy
\be \label{renewal}
\frac{\eta^+(N)}{N^{1-1/\alpha}} \stackrel{\D}{\longrightarrow} \tau^{1/\alpha-1}, \quad \frac{\eta^-(N)}{N^{1-1/\alpha}} \stackrel{\D}{\longrightarrow} \tau^{1/\alpha-1} \qquad \mbox{under } \Pt,
\ee
where $\tau$ is a stable r.v. with index $1-1/\alpha$ that is the weak limit of $\frac{\Theta_n^+}{n^{\alpha/(1-\alpha)}}$. We conclude the proof if check the uniform integrability of $\sqrt{\frac{N^{1-1/\alpha}}{\eta^+(N)+1}}$ and the same for $\eta^-(N)$. For any $0 < x \le N^{\frac12-\frac{1}{2\alpha}}$, we have
\beaa
&& \Pt \Biggl \{\sqrt{\frac{N^{1-1/\alpha}}{\eta^+(N)+1}} \ge x \Biggr \} = \Pt \Bigl \{ \eta^+(N) \le [ x^{-2} N^{1-1/\alpha}] - 1 \Bigr \} \\
&=& \Pt \Bigl \{ \Theta^+_{[ x^{-2} N^{1-1/\alpha} ]} > N \Bigr \} \le \Pt \Bigl \{ \Theta^+_k >  (x^2 k)^{\frac{1}{1-1/\alpha}} \Bigr \},
\eeaa
with $k:=[ x^{-2} N^{1-1/\alpha} ] \ge 1$. The last probability can be estimated by the following analog of the Chebyshev inequality attributed by Nagaev~\cite{Nagaev} to Tkachuk (1977): if $X_n$ are i.i.d. r.v.'s and $X_1 \in \mathcal{DN}(\gamma)$ for $0< \gamma < 1$, then there exist $c, K >0$ such that
$$\P \{X_1 + \dots X_n > R n^{1/\gamma}\} \le c R^{-\gamma}$$
for all $n$ and $R \ge K$. Then the uniform integrability follows as $x^{-2}$ is integrable at infinity.

 {\it Case 2: $S_1$ is right-continuous.} Denote $$r_N:=\Pt \Bigl \{ \min \limits_{1 \le k \le \eta(N)} \Psi_k  > 0 \Bigr \}, \quad \bar{r}_N:=\Pt \Bigl \{ \min \limits_{1 \le k \le \eta(N)} \Psi_k  \ge 0 \Bigr \}$$ and use \eqref{indep1} to replace \eqref{key} by the appropriate inequality; then get an analog of \eqref{keyk} with ``$=$'' and $c_1$ replaced by ``$\le$'' and $\bar{c_1}$, respectively, and by the uniform integrability conclude with $r_N \lesssim N^{\frac12 - \frac{1}{2\alpha}}$. The same argument implies $N^{\frac12 - \frac{1}{2\alpha}} \lesssim \bar{r}_N$, and since $$\bar{r}_N \ge r_N \ge \P \{\psi_1 >0 \} \bar{r}_N,$$ we obtain $\bar{r}_N \asymp  r_N \asymp N^{\frac12 - \frac{1}{2\alpha}}$.
\end{proof}

Let us prove \eqref{c} in the remark. By Lemma~\ref{IID}, it is true that $c_2:=\lim \limits_{n \to \infty} n^{1/2} \P \{\theta_1 >n\} = \sqrt{\frac{8}{\pi}} \frac{\sigma}{\E |S_1|}$ while $c_1=\sqrt{\frac{1}{\pi}}$ as $\Psi_n$ is continuous and symmetric. It remains to compute the limit in \eqref{keyk} using the uniform integrability and \eqref{renewal}, which gives $\eta(N)/N^{1/2} \stackrel{\D}{\to} c_2^{-1} \sqrt{\frac{2}{\pi}}|\mathcal{N}|$ for a standard normal r.v.  $\mathcal{N}$.

\subsection{Positivity of integrated bridges} \label{SSEC BRIDGE}
Let us show how Proposition~\ref{INDEP} can be used to obtain the asymptotics of $$p^*_N=\P \Bigl \{ \min \limits_{1 \le k \le N} A_k  > 0 \bigl | \bigr. S_N = 0\Bigr \}$$ as $N \to \infty$ along $\mathcal{D}_{S_1}$. Recall we assumed that $S_n$ is centered and integer-valued. Let $d$ be the maximal positive integer such that $\P\{S_1 \in d \Z \}= 1$, and let $h$ be the maximal step of $S_1/d$, that is the maximal positive integer such that there exists an $0 \le a \le h-1$ satisfying $\P\{S_1 \in d(a + h \N)\} = 1$. Then $\mathcal{D}_{S_1} \subset h \N$ and $h \N \setminus \mathcal{D}_{S_1}$ is finite.

We need to modify the definition of the regeneration moments considered in Sec.~\ref{PARTITION}. Define the moments of {\it leaving} zero as $\Theta^*_0:=\min\{n \ge 0: S_{n+1} \neq 0\}$ and $\Theta^*_{k+1}:= \min\{n > \Theta_k: S_n = 0, S_{n+1} \neq 0 \}$ for $k \ge 0$, and introduce $\theta^*_k, \psi^*_k, \Psi^*_k, \eta^*(N)$ accordingly. Put $\P^*(\cdot):=\P(\cdot | S_1 \neq 0)$. The following result is completely analogous to Lemma~\ref{IID} and is essentially proved in \cite{IIRW} while the local asymptotics are due to Kesten~\cite{Kesten}.
\begin{lem1'} \label{IID'}
Let $S_n$ be a centered random walk. Then $(\theta^*_n, \psi^*_n)_{n \ge 1}$ are i.i.d. and $(\theta^*_1, \psi^*_1) \stackrel{\D}{=} (\theta^*_1, -\psi^*_1)$. If $Var(S_1)=: \sigma^2 < \infty$, then $\P^*\{\theta^*_1 = hn\} \sim \sqrt{\frac{h}{2 \pi}} \frac{\sigma}{\P\{S_1 \neq 0\} d} n^{-3/2}$ as $n \to \infty$.
\end{lem1'}

We are ready to prove Proposition~\ref{BRIDGE} from Sec. 1 on the asymptotics of $p^*_{hN}$. Similarly to \eqref{reduction}, write
$$\P \Bigl \{ \min \limits_{1 \le k \le hN} A_k  > 0, S_{hN} = 0 \Bigr \} \le \P \Bigl \{ \min_{1 \le k \le \eta^*(hN)} A_{\Theta^*_k}  > 0, A_1 \neq 0, S_{hN} =0  \Bigr \}$$ (which is the equality when $S_1$ is right-continuous) and then condition on the number of returns to zero $\eta^*(hN)$ as in \eqref{condit} to get
$$\P \Bigl \{ \min \limits_{1 \le k \le hN} A_k  > 0, S_{hN} = 0 \Bigr \} / \P\{S_1 \neq 0\}  \le \sum_{n=1}^\infty \P^* \Bigl \{ \min_{1 \le k \le n} \Psi^*_k  > 0, \Theta^*_n = hN  \Bigr \}=:r_N.$$ By Proposition~\ref{INDEP} we get the following analog of \eqref{keyk}:
$$r_N \le (\bar{c}_1 + o(1))\sum_{n=1}^\infty n^{-1/2} \P^* \bigl \{ \Theta^*_n = hN  \bigr \} .$$

Apply Lemma~$1^\prime$ to use the result by Doney~\cite{Doney2} on local large deviation probabilities that states $\P^* \bigl \{ \Theta^*_n = hN  \bigr \} \sim n \P^* \bigl \{ \theta^*_1 = hN  \bigr \}$ as $N \to  \infty$ uniformly in $n=o(\sqrt{N})$. Then the contribution of the terms with $n=o(\sqrt{N})$ is $o(N^{-3/4})$, implying
$$r_N \le (\bar{c}_1 +o(1)) \varlimsup_{\varepsilon \to 0+} \sum_{n= \varepsilon \sqrt{N}}^\infty n^{-1/2} \P^* \bigl \{ \Theta^*_n = hN  \bigr \} + o(N^{-3/4}).$$ Once we have bounded $\sqrt{N}/n$ away from zero, the local limit theorem gives
\beaa
\lim_{\varepsilon \to 0+} \lim_{N \to \infty} N^{3/4} \sum_{n= \varepsilon \sqrt{N}}^\infty n^{-1/2} \P^* \bigl \{ \Theta^*_n = hN  \bigr \} &=&  \lim_{\varepsilon \to 0+} \frac{1}{\sqrt{N}} \sum_{n=\varepsilon \sqrt{N}}^\infty \Bigl(\frac{n}{\sqrt{N}}\Bigr)^{-5/2} n^2 \P^* \bigl \{ \Theta^*_n = hN  \bigr \} \\
&=& \lim_{\varepsilon \to 0+} \frac{1}{\sqrt{N}} \sum_{n=\varepsilon \sqrt{N}}^\infty \Bigl(\frac{n}{\sqrt{N}}\Bigr)^{-5/2} h g\Bigl(\frac{hN}{n^2}\Bigr)\\
&=& h^{1/4} \E \tau^{-5/4},
\eeaa
where $g$ is the density of a strictly stable r.v. $\tau$ with index $1/2$ that is the weak limit of $\Theta^*_n/n^2$.

Thus $r_N \lesssim N^{-3/4}$ and by Gnedenko's local limit theorem, $p^*_{hN} \lesssim N^{-1/4}$. Now assume that $S_1$ is right-continuous to get the
estimate in the other direction. Condition on the $(hN + 1)$st step of the walk to get $p^*_{h N} \ge (\P\{ S_1 = 1\})^2 \bar{p}^*_{hN}$, where $\bar{p}^*_n$ is defined as $p^*_n$ with ``$>$'' replaced by ``$\ge$''. Arguing as above we get $\bar{p}^*_{hN} \gtrsim N^{-1/4}$ which implies that $p^*_{hN} \asymp \bar{p}^*_{hN} \asymp N^{-1/4}$.

\section{Areas and lengths of excursions of asymptotically stable random walks} \label{SEC EXCURSIONS}

This section gathers the preliminary results needed to prove Theorem~\ref{SHARP}. Sec. 3.1 reviews limit theorems on the shape of the trajectories of conditionally positive asymptotically stable random walks. We explain the method of proofs and apply it to get a two-dimensional version, which is used later in Sec.~\ref{SEC SHARP} to describe the bivariate walk $(\Theta_k, \Psi_k)$ conditioned on its second component staying positive.

\subsection{Conditional limit theorems for random walks}

{\it Results and methods.} Let $S_n$ be a random walk such that the first descending ladder moment $T=\min \{k \ge 1: S_k<0 \} < \infty$  a.s. Bolthausen~\cite{Bolt} showed that if $\E S_1 = 0$ and $Var(S_1) = \sigma^2 < \infty$, then
\be \label{boltha}
\mbox{Law} \Bigl( \frac{S_{[n \cdot]}}{n^{1/2}} \Bigl | \Bigr. T \ge n \Bigr)\stackrel{\D}{\longrightarrow} \mbox{Law} \bigl ( \sigma W_+ (\cdot) \bigr )
\ee
in the Skorokhod space $D[0,1]$ as $n \to \infty$, where $W_+$ is a Brownian meander on $[0,1]$ defined below in terms of a standard Brownian motion $W$.

The proof of \cite{Bolt} is based on the following insightfully simple observation. For any $f:[0, \infty) \to \R$ define $\tau_f:=\inf\{ t \ge 0: f(s+t) \ge f(t) \mbox{ for } 0 \le s \le 1 \}$, where $\inf_\varnothing :=\infty$, and $\Gamma(f)(\cdot):=f(\cdot + \tau_f) - f(\tau_f)$ if $\tau_f < \infty$ and $\Gamma(f): \equiv 0$ if otherwise. Then
\be \label{pathtrans}
\mbox{Law} \bigl( S_{[n \cdot]} \bigl | \bigr. T \ge n \bigr) = \mbox{Law} \bigl( \Gamma (S_{[n \cdot]}) \bigr).
\ee
Bolthausen~\cite{Bolt} essentially showed that $\P\{\tau_W < \infty\}=1$ and $\Gamma$ considered as a mapping $C[0, \infty) \to C[0,1]$ is measurable and continuous $\P \{ W \in \cdot \, \}$-a.s. (Wiener measure).
By the linear interpolation, \eqref{boltha} with $W_+ = \Gamma(W)$ immediately follows from the invariance principle in $C[0, \infty)$ and the continuous mapping theorem, see Billingsley~\cite[Sec. 2]{Bil}.

Shimura~\cite{Shimura} used the same method to prove weak convergence of excursions. For any $f:[0, \infty) \to \R$ define $\delta_f:=\inf\{ t \ge 0: f(t)  < 0 \}$ and $\Lambda(f)(\cdot):= \bigl( f( \cdot  \wedge \delta_f) , \, \delta_f \bigr)$. \cite{Shimura} proved that $\P\{\delta_{W_+} < \infty\}=1$ and $\Lambda \Gamma$ considered as a mapping $D[0, \infty) \to D[0, \infty) \times \R$ is measurable and continuous $\P \{ W \in \cdot \, \}$-a.s.; Shimura actually checked continuity along step functions which is sufficient as they are dense in $D[0, \infty)$. Hence under assumptions $\E S_1=0$ and $Var(S_1) = \sigma^2 < \infty$, the continuous mapping theorem implies Shimura's main result
\be \label{Shim}
\mbox{Law} \Bigl( \Bigl(\frac{S_{[n \cdot \wedge T]}}{n^{1/2}}, \frac{T}{n} \Bigr) \Bigl | \Bigr. T \ge n \Bigr) \stackrel{\D}{\longrightarrow} \mbox{Law}  \bigl(\sigma W_+ (\cdot \wedge \delta_{W_+}), \delta_{W_+} \bigr)
\ee
in $D[0, \infty) \times \R$. Now define rescalings $\widehat{\Lambda}_a (f)(\cdot):= \delta_f^{-1/a} f( \cdot \, \delta_f) $, then $\widehat{\Lambda}_a
\Gamma: D[0, \infty) \to D[0, \infty)$ is continuous $\P \{ W \in \cdot \, \}$-a.s. for any $a>0$. As trajectories of $W$ are continuous, $\widehat{\Lambda}_a \Gamma$ is also a.s. continuous as a mapping $D[0, \infty) \to D[0, 1]$, and we restate \eqref{Shim} as
\be \label{Shimura}
\mbox{Law} \Bigl( \Bigl(\frac{S_{[T \cdot]}}{T^{1/2}}, \frac{T}{n} \Bigr) \Bigl | \Bigr. T \ge n \Bigr) \stackrel{\D}{\longrightarrow} \mbox{Law} \bigl(\sigma W_{ex} (\cdot), \delta_{W_+} \bigr),
\ee
in $D[0,1] \times \R$, where $W_{ex}= \widehat{\Lambda}_2 \Gamma(W)$ is a standard Brownian excursion on $[0,1]$. Note that $W_{ex}$ is independent with the length $\delta_{W_+}$ of the excursion of $W_+=\Gamma(W)$ while $\P\{\delta_{W_+} \ge x\} = x^{-1/2}$ for $x \ge 1$, see Bertoin~\cite[Ch. VIII.4]{Bertoin}.

A further refinement is due to Doney~\cite{Doney} who essentially proved that $\P\{ \delta_{S_+} < \infty\} = 1$ and $\Gamma: D[0, \infty) \to D[0, 1]$ and $\Lambda \Gamma: D[0, \infty) \to D[0, \infty) \times \R$ are continuous $\P \{ S \in \cdot \, \}$-a.s. for any strictly stable centered process $S$ with index $1 < \alpha \le 2$. As above, it suffices to check continuity along step functions.  Then $\widehat{\Lambda}_\alpha\Gamma: D[0, \infty) \to D[0, 2]$ is $\P \{ S \in \cdot \, \}$-a.s. continuous as $S_{ex}=\widehat{\Lambda}_\alpha\Gamma(S)$ is constant for $t \ge 1$.

Now assume that $S_n/ (n^{1/\alpha} l(n)) \stackrel{\D}{\to} S(1)$ for some slowly varying function $l(n)$. We restate the result of Doney as we did above with Shimura's \eqref{Shim} in the form
\be \label{Doney}
\mbox{Law} \Bigl( \Bigl(\frac{S_{[T (\cdot \wedge 1)]}}{T^{1/\alpha} l(T)}, \frac{T}{n} \Bigr) \Bigl | \Bigr. T \ge n \Bigr) \stackrel{\D}{\longrightarrow} \mbox{Law} \bigl( S_{ex} (\cdot), \delta_{S_+} \bigr)
\ee
in $D[0,2] \times \R$. Here we have used the fact that $l(T)/l(n) \stackrel{\P}{\to} 1$, which follows as $l(cx)/l(x) \to 1$ uniformly over any interval by the Karamata theorem. As before, $S_{ex}$ is independent with $\delta_{S_+}$, and $\P\{\delta_{S_+} \ge x\} = x^{1/\alpha-1}$ for $x \ge 1$, see Bertoin~\cite[Ch. VIII.4]{Bertoin}.

Assume that $S$ has negative jumps (recall that our main problem concerns spectrally negative $S_1$). \eqref{Doney} is sufficient for the further consideration in Sec.~\ref{SEC SHARP} but it is not hard to give it a slight improvement by proving convergence in $D[0,1] \times \R$. Indeed,
it is easy to check that the restriction from $[0,2]$ on $[0,1]$ is continuous at any point of $$\bigl \{f \in D[0,2]: f(t) \ge 0 \mbox{ on } [0,1), f(1-)>0, f(t) \equiv const < 0 \mbox{ on } [1,2] \bigr\}.$$ It now suffices to prove that
\be \label{no atom}
\P\{S_{ex}(1-)=0\}=0, \quad \P\{S_{ex}(1)=0\}=0
\ee
which means that the overshoot and the undershoot of an excursion are positive a.s.

Consider the walk $S_n:=S(n)$ and write
\bea
\P \Bigl\{ \frac{S_T}{T^{1/\alpha}} \le z \Bigl| \Bigr. \, T \ge n \Bigr\} &=& \sum_{k=n}^\infty \int_0^\infty \P \Bigl\{ \frac{S_{k-1}}{k^{1/\alpha}} \in dx \Bigl| \Bigr. \, S_1 \ge 0, \dots, S_{k-1} \ge 0 \Bigr\} \notag \\
&& \, \times \frac{\P \{ T \ge k-1\} \P \{ S_1 \le -(x+z^-) k^{1/\alpha}\}}{\P\{ T \ge n \}}, \label{stable exc}
\eea
where $z^-:=-(z \wedge 0)$. Recall that $T \in \mathcal{DN}(1-\rho)$, where $\rho:=\P(S_1 \ge 0)$ is the positivity parameter that satisfies (Zolotarev~\cite[Eq. (2.2.30)]{Zol}) $1 - 1/\alpha < \rho \le 1/ \alpha$ as $S_1$ has negative jumps. Now use the conditional local limit Theorem 3 by Vatutin and Wachtel~\cite{Vatutin} on the weak convergence to the endpoint $S_+(1)$ of a stable meander $S_+$ and its refined version Theorem 4 for small $x$ combined with their Theorem 7 that claims $p_+(1, x):= \P (S_+(1) \in dx) /dx \sim c x^{\alpha \rho}$ as $x \to 0+$. By \eqref{Doney} we already know that $S_T/T^{1/\alpha} \stackrel{\D}{\to} S_{ex}(1)$, hence for any point $z$ of continuity of $S_{ex}(1)$,
$$\P \{S_{ex}(1) \le z \} = \lim_{n \to \infty} \frac{c'}{n} \sum_{k=n}^\infty \Bigl( \frac{k}{n} \Bigr)^{\rho-2} \int_0^\infty p_+(1, x) (x+z^-)^{-\alpha} dx
= \frac{c'}{1-\rho} \E (S_+(1)+z^-)^{-\alpha}.$$ The latter is continuous and differentiable so $S_{ex}(1)$ has a density. The argument above also implies $\P\{S_{ex}(1-) \in dx\} = c'/(1-\rho) x^{-\alpha} p_+(1, x)$. Thus \eqref{no atom} is proved.

{\it Extension to two dimensions.} We stress that all the mentioned results follow from the functional stable limit theorems with the use of the continuous mapping theorem. Let us give a little strengthening to \eqref{boltha}. First extend the definitions of $\tau_f$ and $\Gamma$ to the higher dimension: for an $\mathbf{f}=(f^{(1)}, f^{(2)})$, put $\tau_{\mathbf{f}}:= \tau_{f^{(2)}}$ and $\Gamma(\mathbf{f}):=\mathbf{f}(\cdot + \tau_{\mathbf{f}}) - \mathbf{f}(\tau_{\mathbf{f}})$.

Let $\mathbf{S}_n=(S_n^{(1)}, S_n^{(2)})$ be a bivariate random walk that satisfies
\be \label{2D conv}
\Bigl ( \frac{S_{[n \cdot]}^{(1)}}{n^{1/{\alpha_1}} l_1(n)}, \frac{S_{[n \cdot]}^{(2)}}{n^{1/{\alpha_2}} l_2(n)} \Bigr) \stackrel{\D}{\longrightarrow}  \mathbf{S}( \cdot)
\ee
in $D^2[0, \infty)$ for some bivariate stochastic process $\mathbf{S}$, slowly varying functions $l_1(n), l_2(n)$, and $0 < \alpha_1, \alpha_2 \le 2$. By Resnick and Greenwood~\cite{Resnick}, \eqref{2D conv} is equivalent to existence of the finite positive
\be \label{2D tails}
\lim_{n \to \infty} n \P \bigl \{ \epsilon_1 S_1^{(1)} > x n^{1/{\alpha_1}} l_1(n), \, \epsilon_2 S_1^{(2)} > y n^{1/{\alpha_2}} l_2(n)\bigr\}
\ee
for all $\epsilon_1,\epsilon_2 \in \{-1, 1\}$ and $x,y \ge 0$ such that $x+y>0$. \cite{Resnick} also shows that \eqref{2D conv} is equivalent to the weak convergence of the one-dimensional distributions at $t=1$. The limit random vector $\mathbf{S}(1)$ is sometimes called bivariate stable with indices $\alpha_1, \alpha_2$ as its independent copies $\mathbf{S}'(1), \mathbf{S}''(1)$ satisfy $$a_1 \mathbf{S}'(1) + a_2 \mathbf{S}''(1) \stackrel{\D}{=} \bigl ( (a_1^{\alpha_1} + a_2^{\alpha_1})^{1/{\alpha_1}} S^{(1)}(1), \, (a_1^{\alpha_2} + a_2^{\alpha_2})^{1/{\alpha_2}} S^{(2)}(1) \bigr)$$ for any $a_1, a_2 >0$. \cite{Resnick} gave a complete characterization of such bivariate distributions.

By the $\P\{\mathbf{S} \in \cdot \}$-a.s. continuity of $\Gamma: D^2[0, \infty) \to D^2[0, 1]$ we get
\be \label{bolt2D}
\mbox{Law} \Bigl( \Bigl(\frac{S^{(1)}_{[n \cdot]}}{n^{1/\alpha_1} l_1(n)}, \frac{S^{(2)}_{[n \cdot]}}{n^{1/\alpha_2} l_2(n)}\Bigr) \Bigl | \Bigr. T^{(1)} \ge n \Bigr) \stackrel{\D}{\longrightarrow} \mbox{Law} \bigl( \mathbf{S}_+ (\cdot) \bigr)
\ee
in $D^2[0, 1]$, where $T^{(1)}$ is the first ladder moment of $S^{(1)}$ and $\mathbf{S}_+  := \Gamma(\mathbf{S})$. A simple consideration of \eqref{bolt2D} shows that it also holds true if $T^{(1)}$ is replaced by the first {\it strict} ladder moment.

\subsection{Areas of cycles}

The first statement of this section generalizes Proposition~1 by Vysotsky~\cite{IIRW} that covers the case $\alpha =2$. The second statement will be used to describe the last incomplete cycle. We stress that the long cycles of a random walk from $\mathcal{R}_\alpha$ behave very differently for $\alpha =2 $ and $1 < \alpha <2$, and the same is true for the excursions. In the first case the whole cycle is essentially either positive or negative, while for $1 < \alpha <2$ the walk spends positive parts of time in both half-planes, see \eqref{shape}. For $\alpha=2$, a typical excursion is continuous while for $\alpha < 2$, a typical excursion looks like a meander and then it takes only one big step to change its sign, see \eqref{stable exc}.

\begin{lem} \label{CYCLE}
Let $S_n$ be a random walk such that $S_1 \in \mathcal{R}_\alpha$ for some $1 < \alpha \le 2$. Then for any $\epsilon \in \{-1,1\}$ and $s,t \ge 0$ such that $s+t>0$ there exists a finite positive $$F^{sign(\epsilon)}(s,t):= \lim_{n \to \infty} n^{1 - \frac{1}{\alpha}} \P \{ \theta_1 > s n, \epsilon \psi_1 > t n^{1 +\frac{1}{\alpha}}\}.$$
\end{lem}
Thus \eqref{2D tails} holds true and we conclude that the bivariate stable limit theorem \eqref{2D conv} holds for the walk $(\Theta_n, \Psi_n)$.
\begin{proof}
For $\alpha = 2$, this is the result of Proposition~1 from \cite{IIRW}, which actually proves it only for a right-exponential $S_1$ but the right-continuous case should be considered {\it exactly} in the same way. Assume now that $1 < \alpha <2$.

By Lemma~\ref{IID}, we should consider only $\epsilon = 1$. $S_1 \in \mathcal{R}_\alpha$ implies $T \in \mathcal{DN}(1-1/\alpha)$ and we denote $$c_3:=\lim \limits_{n \to \infty} n^{1-1/\alpha} \Pt\{ T > n\} = \lim \limits_{n \to \infty} n^{1-1/\alpha} \frac{\P\{ T > n\}}{\P\{S_1>0\}}, \quad c_4:=\lim \limits_{n \to \infty} n^\alpha \P\{ S_1 < -n\}.$$ We first claim that
\be \label{shape}
\lim_{\varepsilon \to 0+} \varlimsup_{n \to \infty}  n^{1-\frac{1}{\alpha}} \P \bigl\{ \theta_1 > n ,  \theta_+ < \varepsilon n \bigr\} = 0.
\ee
Heuristically this means that a long cycle starts with a positive excursion of comparable length. Use \eqref{symm} and $\theta_+ = T-1$ under $\Pt$ to write
\beaa
\lim_{\varepsilon \to 0+} \varlimsup_{n \to \infty} n^{1-\frac{1}{\alpha}} \P \bigl\{ \theta_1 > n ,  \theta_+ < \varepsilon n \bigr\} &\le& \lim_{\varepsilon \to 0+} \varlimsup_{n \to \infty} n^{1-\frac{1}{\alpha}}  \P \bigl\{ \theta_- > (1-\varepsilon) n ,  \theta_+ < \varepsilon n \bigr\}\\
&=& \lim_{\varepsilon \to 0+} \varlimsup_{n \to \infty} n^{1-\frac{1}{\alpha}} \Pt \bigl\{ T > (1-\varepsilon) n , S_T \le -(\varepsilon^{1/2} n)^{1/\alpha}, \theta_- < \varepsilon n \bigr\}\\
&\le& c_3 \lim_{\varepsilon \to 0+} \varlimsup_{n \to \infty} \P \bigl\{ T'((\varepsilon^{1/2} n)^{1/\alpha}) < \varepsilon n \bigr\} \\
&=& c_3 \lim_{\varepsilon \to 0+} \P \bigl\{ T''(\varepsilon^{1/(2\alpha)}) < \varepsilon \bigr\} =0,
\eeaa
where $T'(u) := \min \{k \ge 0 : S_k > u\}$ and $T''(u) := \inf \{r \ge 0 : S(r) > u\}$, and we used \eqref{no atom} in the second line and the self-similarity $T''(u) \stackrel{\D}{=} u^\alpha T''(1)$ in the fourth line.

Consider $s \neq 0$, then by the obvious change of variables it suffices to take $s=1$. For an $\varepsilon \in (0,1/2)$, write
\bea
&& \lim_{n \to \infty} n^{1 - \frac{1}{\alpha}} \P \{ \theta_1 > n, \psi_1 > t n^{1 +\frac{1}{\alpha}}, \theta_+ \ge \varepsilon n\} \notag\\
&=& \lim_{n \to \infty} n^{1 - \frac{1}{\alpha}} \Pt \{ T > \varepsilon n\} \lim_{n \to \infty} \P \bigl\{\theta_1 > n, \psi_1 > t n^{1 + \frac{1}{\alpha}} \bigl| \bigr. \, T \ge \varepsilon n\bigr\}. \label{new1}
\eea
In the second factor, condition on the parameters of the first positive excursion and write
\bea \notag
&& \lim_{n \to \infty} \P \bigl\{\theta_1 > n, \psi_1 > t n^{1 + \frac{1}{\alpha}} \bigl| \bigr. \, T \ge \varepsilon n\bigr\} \\
&=& \label{stable area 1} \lim_{n \to \infty} \int_1^\infty \int_0^\infty \int_0^\infty f_n^+ \bigl((\varepsilon x )^{1/\alpha} z, 1 - \varepsilon x, t - (\varepsilon x )^{1+ 1/\alpha} y \bigr) P_n^{(\varepsilon)} (dx,dy,dz),
\eea
where
$$P_n^{(\varepsilon)} (dx,dy,dz):= \P \Bigl\{ \frac{T}{\varepsilon n } \in dx, \frac{A_T}{T^{1+\frac{1}{\alpha}}} \in dy, \frac{S_T}{T^{\frac{1}{\alpha}}} \in -dz \Bigl| \Bigr. \, T \ge \varepsilon n \Bigr\}$$
and
$$f_n^+(u, v, w) := \P \Bigl\{T'(u n^{1/\alpha} )-1 > n v, \sum_{i=1}^{T'(u n^{1/\alpha} )-1} (S_i - u n^{1/\alpha}) > w n^{1 + 1/\alpha}\Bigr\}.$$

It clear that for any $u > 0$ and $v,w$ it holds that $$\lim_{n \to \infty} f_n^+(u, v, w) = f^+(u, v, w):=\P \Bigl\{T''(u) \ge v , \, \int_0^{T''(u)} (S(r) - u) dr \ge w \Bigr\}.$$ We claim that this convergence is uniform in $(u, v, w) \in [\delta, \infty) \times \R^2$ for any $\delta >0$. As $f_n^+(u,v,w)=f_{n u^\alpha}^+(1,u^{-\alpha} v, u^{-\alpha-1} w)$ and $f^+(u,v,w)=f^+(1,u^{-\alpha} v, u^{-\alpha-1} w)$ by self-similarity of $S$, we should check that the convergence is uniform in $(v, w) \in \R^2$ for $u=1$. This statement just a little improvement of the standard fact that the distribution functions of weakly convergent r.v.'s converge uniformly if the limit distribution is continuous. We prove the bivariate uniformness by showing that the r.v.'s $T''(1)$ and $\int_0^{T''(1)} (S(r) - u) dr$ are continuous. The first clearly is, say, since $T''(u)$ is a stable subordinator with index $1/\alpha$ as $S$ does not have positive jumps. For the second, use that $\int_0^x (S(r) - u) dr$ and $S(x)$ are jointly stable and, consequently, have a joint density for any $x >0$.

Thus the integrands in \eqref{stable area 1} converge uniformly in $(x,y,z) \in [1, \infty) \times [0, \infty) \times [\delta, \infty)$. Further, \eqref{Doney} ensures $P_n^{(\varepsilon)} \stackrel{\D}{\to} P$ in $[1, \infty) \times \R^2_+$ for any fixed $\varepsilon$, where
$$P(dx,dy,dz):=d (-x^{\frac{1}{\alpha}-1} ) \P \Bigl\{ \int_0^1 S_{ex}(s) ds \in dy, S_{ex}(1) \in -dz \Bigr\}.$$ As we seen in \eqref{no atom}, $S_{ex}(1)$ does not have an atom at zero, so \eqref{shape}, \eqref{new1} and \eqref{stable area 1} imply
\bea
F^+(1,t) &=& c_3 \lim_{\varepsilon \to 0+}  \varepsilon^{\frac{1}{\alpha}-1} \int_1^\infty \int_0^\infty \int_0^\infty f^+ \bigl((\varepsilon x )^{1/\alpha} z, 1 - \varepsilon x, t - (\varepsilon x )^{1+ 1/\alpha} y \bigr) P(dx, dy, dz) \notag \\
&=& c_3  \iiint \limits_{\R_+^3} f^+ \bigl(x^{1/\alpha} z, 1 - x, t - x^{1+ 1/\alpha} y \bigr) P(dx, dy, dz) \label{stable area 2}.
\eea
This expression is finite as by \eqref{symm},
$$F^+(1,t) \le  \varlimsup_{n \to \infty} n^{1 - \frac{1}{\alpha}} \P \bigl\{ \theta_+ + \theta_- > n\bigr\} \le 2 \lim_{n \to \infty} n^{1 - \frac{1}{\alpha}} \P \bigl\{ \theta_+ > n/2 \bigr\} = 2 c_3.$$

It remains to consider $s=0, t>0$ and it suffices to take $t=1$. With $\psi_1 \le A_T$ in mind, we have
\beaa
\lim_{\varepsilon \to 0+} \varlimsup_{n \to \infty} n^{1 - \frac{1}{\alpha}} \P \{ \theta_+ < \varepsilon n, \psi_1 > n^{1 +\frac{1}{\alpha}}\} &\le& \lim_{\varepsilon \to 0+} \varlimsup_{n \to \infty} n^{1 - \frac{1}{\alpha}} \Pt \{ T < \varepsilon n, A_{T-1} >  n^{1 +\frac{1}{\alpha}}\}\\
&\le& \lim_{\varepsilon \to 0+} \varlimsup_{n \to \infty} n^{1 - \frac{1}{\alpha}} \Pt \{ \max_{1 \le k < \varepsilon n} S_k > \varepsilon^{-1} n^{\frac{1}{\alpha}}\} = 0,
\eeaa
and then argue as above in \eqref{stable area 1} and \eqref{stable area 2} to get
$$F^+(0,1) = c_3  \iiint \limits_{\R_+^3} f^+ \bigl(x^{1/\alpha} z, 0, 1 - x^{1+ 1/\alpha} y \bigr) P(dx, dy, dz).$$ The last  expression is finite as $A_T \in \mathcal{DN}(\frac{\alpha-1}{\alpha+1})$ by Corollary 2 and Example 6 by Doney~\cite{Doney}.

\end{proof}

\subsection{Areas of incomplete cycles}

\begin{lem} \label{LAST 1}
Let $S_n$ be a random walk such that $S_1 \in \mathcal{R}_\alpha$ for some $1 < \alpha \le 2$. Then for any $s,t >0$ there exists a finite positive
$$F(s,t):=\lim_{n \to \infty} n^{1-\frac{1}{\alpha}} \Pt \Bigl \{ \theta_1 \ge s n, A_{sn} > -tn^{1+\frac{1}{\alpha}} \Bigr \},$$ and this convergence is uniform in $(s,t) \in [\varepsilon, \infty) \times [0, \infty)$ for any $\varepsilon >0$. Moreover, $F(s,t)$ is continuous on $\R_+^2$.
\end{lem}
\begin{rem*}
For any $s,t >0$ it holds that $F(s,t)=c_3 s^{\frac{1}{\alpha}-1} \bigl(1 + G_\alpha(t s^{-1-\frac{1}{\alpha}}) \bigr)$ for some increasing function $G_\alpha: (0, \infty) \to [0,1]$, with $G_2(x)$ equal to the distribution function of the area of a Brownian meander $\sigma W_+$.
\end{rem*}

\begin{proof}
The remark is obvious by the change of variables and
$$F(s,t) = \lim_{n \to \infty} n^{1-\frac{1}{\alpha}} \Pt \bigl \{ \theta_+ \ge sn \bigr \} + \lim_{n \to \infty} n^{1-\frac{1}{\alpha}} \Pt \Bigl \{ \theta_+ < sn, \theta_1 \ge sn, A_n > -tn^{1+\frac{1}{\alpha}} \Bigr \}.$$ Let us put $s=1$.

{\it Case $1 < \alpha <2$.} We literally repeat the proof of Lemma~\ref{CYCLE} to get
$$F(1,t) = c_3  \iiint \limits_{\R_+^3} f \bigl(x^{1/\alpha} z, 1 - x,  -t- x^{1+ 1/\alpha} y \bigr) P(dx, dy, dz)$$ with
$$f(u, v, w):=\P \Bigl\{T''(u) \ge v^+ , \, \int_0^{v^+} (S(r) - u) dr \ge w \Bigr\}.$$ The continuity of $F(1,t)$ follows from the same of $f(u, v,w)$ and the theorem of dominated convergence. As $F(1,t)$ is bounded, continuous and monotone and the converging functions are uniformly bounded and monotone, the convergence is uniform in $t$.


{\it Case $\alpha =2$.} The main difference with $1 < \alpha <2$ is that \eqref{shape} is no longer true.

We should prove that $c_3 G_2(t) = \lim \limits_{n \to \infty} n^{1/2} \Pt \bigl \{ \theta_+ < n, \theta_1 \ge n, A_n > -tn^{3/2} \bigr \},$ and the goal is to show that the contribution comes only from the negative excursion:
\be \label{last2}
\lim_{n \to \infty} n^{1/2} \Pt \Bigl \{ \theta_+ < n, \theta_1 \ge n, A_n  > -tn^{3/2} \Bigr \} = \lim_{n \to \infty} n^{1/2}  \Pt \Bigl \{ \theta_- \ge n, A_{n + \theta_+} - A_{\theta_+} > -tn^{3/2} \Bigr \}.
\ee
Heuristically this holds as the positive part of such a long cycle with $\theta_+ < n$ is negligible, according to our previous comment on the shape of long cycles for $\alpha =2$.

Let us first find the value of the right-hand side. By \eqref{symm}, we get
\be \label{symm2}
\lim_{n \to \infty} n^{1/2}  \Pt \Bigl \{ \theta_- \ge n, A_{n + \theta_+} - A_{\theta_+} > -tn^{3/2} \Bigr \} = \lim_{n \to \infty} n^{1/2}  \Pt \Bigl \{ \theta_+ \ge n, A_{\theta_+} - A_{\theta_+ - n} < tn^{3/2} \Bigr \}
\ee
as $n \to \infty$, reducing the problem to consideration of the first positive excursion of $S_n$. Then we apply \eqref{Shimura} to get
\bea
\lim_{n \to \infty} n^{1/2} \Pt \Bigl \{ \theta_+ \ge n, A_{\theta_+} - A_{\theta_+ - n} < tn^{3/2} \Bigr \} &=& c_3 \P \biggl \{ \sigma \delta_{W_+}^{3/2} \int_{1-\delta_{W_+}^{-1}}^1 W_{ex}(s) ds < t \biggr \} \notag \\
&=& c_3 \P \biggl \{ \int_0^1 \sigma \delta_{W_+}^{1/2} W_{ex}(s \delta_{W_+}^{-1}) ds < t\biggr \} \notag\\
&=& c_3 G_2(t) \label{last3},
\eea
where we used $\delta_{W_+}^{1/2} W_{ex}(\cdot \delta_{W_+}^{-1}) \stackrel{\D}{=} W_+(\cdot)$, which follows from \eqref{boltha} and \eqref{Shimura}. Due to Janson~\cite{Janson}, the area of Brownian meander has density so $G_2$ is continuous.

It remains to prove \eqref{last2}. Fix a $\delta \in (0,1/2)$ and write
\be \label{dUe}
\bigl \{ \theta_+ < n, \theta_1 \ge n \bigr \} = D \cup E_1,
\ee
where $E_1:=  \bigl \{ (1-\delta) n \le \theta_+ < n \bigr \} \cup \bigl \{ (1-\delta) n \le \theta_- < n \bigr \} \cup \bigl \{ \theta_+ \ge \delta n, \theta_- \ge \delta n \bigr \}$ and $D:= \bigl \{ \theta_+ < \delta n, \theta_- \ge n\bigr \}$. Note that $$\lim_{\delta \to 0+} \lim_{n \to \infty} n^{1/2} \bigl( \Pt \bigl \{ \theta_- \ge n \bigr \} - \Pt (D) \bigr) =0, \quad \lim_{\delta \to 0+} \lim_{n \to \infty} n^{1/2} \Pt (E_1) = 0$$ by continuity in $s$ of $\lim \limits_{n \to \infty} n^{1/2} \Pt \bigl \{ \theta_{+/-} \ge s n \bigr \}$ and the relation $$\lim_{n \to \infty} n^{1/2} \Pt \bigl \{ \theta_+ \ge n, \theta_- \ge n \bigr \} =0,$$ which is Eq. (17) from \cite{IIRW}. Its heuristical meaning is that a long cycle is essentially either positive or negative. Thus $D$ gives the main contribution in \eqref{dUe}.

Further, from \eqref{dUe} we have
\be \label{dUe2}
\bigl \{ \theta_+ < n, \theta_1 \ge n, A_n > -t n^{3/2} \bigr \} \subset D \cap \bigl \{ A_n - A_{\theta_+} > - t n^{3/2} \bigr \} \cup E_1 \cup E_2 \cup E_3,
\ee
where $$E_2:= \bigl \{ \theta_+ < \delta n, \theta_- \ge n, - (t+\delta) n^{3/2} \le A_n - A_{\theta_+} \le -t n^{3/2} \bigr \}, \quad E_3:= \bigl \{ \theta_+ < \delta n, A_{\theta_+} > \delta n^{3/2} \bigr \}.$$ Here $\lim \limits_{\delta \to 0+} \lim \limits_{n \to \infty} n^{1/2} \Pt (E_3) = \lim \limits_{\delta \to 0+} F^+(0,\delta) - F^+(\delta, \delta)= 0$ by the continuity of $F^+(s,t)$ at $(0,0)$ given in Proposition 1 in Vysotsky~\cite{IIRW}. For the same relation for $E_2$, write $$E_2 \subset \Bigl \{ \theta_- \ge n,  A_{n + \theta_+} - A_{(1-\delta)n + \theta_+} - (t+\delta) n^{3/2} \le A_{n + \theta_+} - A_{\theta_+} \le -t n^{3/2} \Bigr \},$$ use the symmetry as in \eqref{symm2} and then argue as in \eqref{last3} to get
$$\lim_{\delta \to 0+} \lim_{n \to \infty} n^{1/2} \Pt (E_2) \le c_3 \P \biggl \{ \sigma^{-1} t  <  \int_0^1 W_+(s) ds < \int_{1-\delta}^1 W_+(s) ds + \sigma^{-1} (t + \delta) \biggr \} =0.$$

Finally, combine
\beaa
&& \bigl \{ \theta_+ < \delta n, \theta_- \ge n, A_{n + \theta_+} - A_{\theta_+} > - t n^{3/2} \bigr \} \\
&\subset& D \cap \bigl \{ A_n - A_{\theta_+} > - t n^{3/2} \bigr \} \\
&\subset& \bigl \{ \theta_- \ge (1-\delta) n, A_{(1-\delta) n + \theta_+} - A_{\theta_+} > - t n^{3/2} \bigr \} \\
\eeaa
with \eqref{last3} and the continuity of $G_2$ to get $$\lim_{\delta \to 0+} \lim_{n \to \infty} n^{1/2} \Bigl( \Pt \bigl \{\theta_- \ge n, A_{n + \theta_+} - A_{\theta_+} > - t n^{3/2} \bigr \} - \Pt \bigl(D \cap \bigl \{ A_n - A_{\theta_+} > - t n^{3/2} \bigr \} \bigr) \Bigr) =0.$$ Together with \eqref{dUe2} and the estimates above this concludes \eqref{last2}.
\end{proof}

\section{Sharp asymptotics of $p_N$} \label{SEC SHARP}
In this section we prove the core result of the paper. Let us first explain the main ideas of the proof. When considering the weak asymptotics of $p_N$ in Proposition~\ref{WEAK} we used the symmetry and successfully disregarded the last incomplete cycle obtaining \eqref{disregard}. For the sharp asymptotics, the incomplete area $A_N - A_{\Theta_{\eta(N)}}$ should be considered and compared with the total area $A_{\Theta_{\eta(N)}}=\Psi_{\eta(N)}$ of the preceding cycles. For such a comparison we will condition on $\eta(N)$, $\Theta_{\eta(N)}$ and $\Psi_{\eta(N)}$ in Step 1 in order to use that the length $\theta_{\eta(N)+1}$  of the last incomplete cycle exceeds $N-\Theta_{\eta(N)}$ and its incomplete area exceeds $-\Psi_{\eta(N)}$. Note that the areas are of the same order since by the classical results of renewal theory (see Feller~\cite[Ch. XIV.3]{Feller}), $N-\Theta_{\eta(N)}$ is of the order $N$. In Step 2 we essentially apply the conditional limit theorem \eqref{bolt2D} to describe the joint distribution of $(\Theta_{\eta(N)}, \Psi_{\eta(N)})$ conditioned that its second component stayed positive. Step 3 is somewhat technical.

{\it Step 1.} Condition on $\eta(N)$ in \eqref{reduction} to obtain
\beaa
\frac{p_N}{\P\{S_1>0\}} &=& \sum_{k=0}^\infty \Pt \Bigl \{ \eta(N) = k, \min \limits_{1 \le i \le k} \Psi_i > 0, A_N >0 \Bigr \} \\
&=& \sum_{k= \varepsilon N^{1-1/\alpha}}^{\varepsilon^{-1} N^{1-1/\alpha}} \Pt \Bigl \{  \Theta_k \le N, \theta_{k+1} > N-\Theta_k, \min \limits_{1 \le i \le k} \Psi_i  > 0, A_N - A_{\Theta_k} > -\Psi_k \Bigr \} \\
&& + R_1(\varepsilon, N) + R_2(\varepsilon, N),
\eeaa
where an $\varepsilon \in (0, 1/2)$ is fixed while $R_1$ and $R_2$ by definition corresponds to $\eta(N) < \varepsilon N^{1-1/\alpha}$ and $\eta(N) > \varepsilon^{-1} N^{1-1/\alpha}$, respectively. Let $S_n'$ be an independent copy of $S_n$. As $\theta_{\eta(N)+1}$ has the order $N$, we write $$\frac{p_N}{\P\{S_1>0\}} = \sum_{k= \varepsilon N^{1-1/\alpha}}^{\varepsilon^{-1} N^{1-1/\alpha}} \Pt \Bigl \{ \Theta_k \le (1-\varepsilon^2) N,  \theta_{k+1} > N-\Theta_k, \min \limits_{1 \le i \le k} \Psi_i > 0, A_{N-\Theta_k}' > -\Psi_k \Bigr \} + R(\varepsilon, N)$$
with $R(\varepsilon, N) := R_1 + R_2 + R_3$ and $R_3=R_3(\varepsilon, N)$ by definition corresponding to $(1-\varepsilon^2) N < \Theta_k < N $.

For each $k$, condition on $(\Theta_k, \Psi_k)$ and rewrite the last formula as
\beaa
&& \frac{p_N}{\P\{ S_1 >0 \}} \\
&=& \sum_{k= \varepsilon N^{1-1/\alpha}}^{\varepsilon^{-1} N^{1-1/\alpha}} \Pt \Bigl \{ \min \limits_{1 \le i \le k} \Psi_i  > 0 \Bigr \}
\int_0^{(1-\varepsilon^2) N} \int_0^\infty \Pt \Bigl \{ \theta_{k+1} > N-x, A_{N-x}'  > -y \Bigr \} \\
&& \times \Pt \Bigl \{ \Theta_k \in dx, \Psi_k \in dy \Bigl | \Bigr. \min \limits_{1 \le i \le k} \Psi_i > 0 \Bigr \}  + R(\varepsilon, N), \\
&=& \frac{1}{N^{1-1/\alpha}} \sum_{k= \varepsilon N^{1-1/\alpha}}^{\varepsilon^{-1} N^{1-1/\alpha}} \Pt \Bigl \{ \min \limits_{1 \le i \le k} \Psi_i  > 0 \Bigr \} \int_0^{\frac{(1-\varepsilon^2) N}{k^{\alpha/(\alpha-1)}}} \int_0^\infty F_N \Bigl(1- x \Bigl ( \frac{k}{N^{1-1/\alpha}} \Bigr)^{\frac{\alpha}{\alpha-1}}, y \Bigl ( \frac{k}{N^{1-1/\alpha}} \Bigr)^{\frac{\alpha+1}{\alpha-1}} \Bigr) \\
&& \times \Pt \Bigl \{ \frac{\Theta_k}{k^{\alpha/(\alpha-1)}} \in dx, \frac{\Psi_k}{k^{(\alpha+1)/(\alpha-1)}} \in dy \Bigl | \Bigr. \min \limits_{1 \le i \le k} \Psi_i  > 0 \Bigr \} + R(\varepsilon, N)
\eeaa
with $$F_n(u,v):= n^{1-1/\alpha} \Pt \Bigl \{ \theta_1 > u n, A_{un} > -v n^{1+1/\alpha} \Bigr \} $$ corresponding to the last incomplete cycle. Let $Q_k(dx, dy)$ denote the conditional probability measure in the last integral. Thinking of the summation as of the integration over the discretization of the Lebesgue measure $\lambda$, we introduce $$U_n(dz):= n^{-1} \delta_0(\{zn \}), \quad P_n(dz, dx,dy):= Q_{z n^{1-1/\alpha}}(dx, dy) U_{n^{1 -1/\alpha}}(dz)$$ and get
\beaa
\frac{p_N}{\P\{ S_1 >0 \}} &=& \int_\varepsilon^{\frac{1}{\varepsilon}} \int_0^{\frac{1-\varepsilon^2}{z^{\alpha/(\alpha-1)}}} \int_0^\infty \Pt \Bigl \{ \min \limits_{1 \le i \le z N^{1 -1/\alpha}} \Psi_i  > 0 \Bigr \} F_N \bigl(1- x z^{\frac{\alpha}{\alpha-1}}, y z^{\frac{\alpha+1}{\alpha-1}} \bigr) P_N(dz, dx, dy) \\
&&+ R(\varepsilon, N).
\eeaa

{\it Step 2.} By Lemma~\ref{LAST 1},
\bea \label{pN3}
\frac{p_N N^{\frac12 - \frac{1}{2\alpha}}}{c_1 \P\{ S_1 >0 \}} &=& \int_\varepsilon^{\frac{1}{\varepsilon}} \int_0^{\frac{1-\varepsilon^2}{z^{\alpha/(\alpha-1)}}} \int_0^\infty  z^{-1/2} F \bigl(1- x z^{\frac{\alpha}{\alpha-1}}, y z^{\frac{\alpha+1}{\alpha-1}} \bigr) P_N(dz, dx, dy) \\
&& + o_\varepsilon(1) + R(\varepsilon, N) N^{\frac12 - \frac{1}{2\alpha}} \notag
\eea
as $N \to \infty$. Further, Lemma~\ref{CYCLE} ensures \eqref{2D tails} that implies \eqref{2D conv}, that is,
$$\Bigl ( \frac{\Psi_n}{n^{(\alpha+1)/(\alpha-1)}}, \frac{\Theta_n}{n^{\alpha/(\alpha-1)}} \Bigr) \stackrel{\D}{\longrightarrow}  \mathbf{S}(1) \mbox{ under } \Pt,$$ where $\mathbf{S}(1)$ is a bivariate stable r.v. (in the sense of Resnick and Greenwood~\cite{Resnick}) with indices $\frac{\alpha-1}{\alpha+1}, \frac{\alpha-1}{\alpha}$. By \eqref{bolt2D}, we have $Q_n \stackrel{\D}{\to} \mbox{Law}(\mathbf{S_+}(1))$ implying $P_N \stackrel{\D}{\to} \lambda|_{[\varepsilon, \varepsilon^{-1}]} \otimes \mbox{Law}(\mathbf{S_+}(1))$ as we are concerned with $z \ge \varepsilon$. The integrand in \eqref{pN3} is continuous a.s. with respect to the limit measure so
$$C_\alpha:=\lim_{N \to \infty} \frac{p_N N^{\frac12 - \frac{1}{2\alpha}}}{c_1 \P\{ S_1 >0 \}} = \int_0^\infty \int_0^{z^{-\frac{\alpha}{\alpha-1}}} \int_0^\infty z^{-1/2} F \bigl(1- x z^{\frac{\alpha}{\alpha-1}}, y z^{\frac{\alpha+1}{\alpha-1}} \bigr) dz \P\{\mathbf{S_+}(1) \in (dx, dy)\}$$
if we check that
\be \label{rest}
\lim_{\varepsilon \to 0+} \varlimsup_{N \to \infty} R(\varepsilon, N) N^{\frac12 - \frac{1}{2\alpha}} = 0.
\ee
We simply the formula for the constant using $F \bigl(1- x z^{\frac{\alpha}{\alpha-1}}, y z^{\frac{\alpha+1}{\alpha-1}} \bigr) = z^{-1} F \bigl(z^{-\frac{\alpha}{\alpha-1}} - x, y \bigr)$ and making the change in the integral:
\be \label{const}
C_\alpha = \int_0^\infty \int_0^u \int_0^\infty F(u-x , y) d (u^{\frac{\alpha-1}{2 \alpha}}) \P\{\mathbf{S_+}(1) \in (dx, dy)\}.
\ee
The right-hand side is finite by Proposition~\ref{WEAK}. Of course this can be checked directly using $F(u-x , y) \le c_3 (u-x)^{\frac{1}{\alpha} -1}$ and the observation that $\mathbf{S}_+^{(1)}(1) \stackrel{\D}{=} \mathbf{S}^{(1)}(1)$, which follows from Proposition~\ref{INDEP}.

{\it Step 3.} It remains to check \eqref{rest} to show that the contribution of $R=R_1+R_2 +R_3$ is negligible. Use \eqref{keyk} in the right-exponential case and use the analogous inequality in the right-continuous case to get
\beaa
R_1(\varepsilon, N) + R_2(\varepsilon, N) &\le& \Pt \Bigl \{ \min \limits_{1 \le k \le \eta(N)} \Psi_k  > 0, \frac{\eta(N)}{N^{1-1/\alpha}} \notin [\varepsilon, \varepsilon^{-1}] \Bigr \} \\
&\le& \frac{\bar{c_1} + o(1)}{N^{\frac12 - \frac{1}{2\alpha}}} \Et \, \I_{[\varepsilon, \varepsilon^{-1}]^c } \biggl(\sqrt{\frac{N^{1-1/\alpha}}{\eta(N)}}\biggr) \sqrt{\frac{N^{1-1/\alpha}}{\eta(N)+1}}.
\eeaa
Then the required estimate for $R_1+R_2$ follows from \eqref{renewal} and the uniform integrability of $\sqrt{\frac{N^{1-1/\alpha}}{\eta(N)+1}}$, which we checked when proved Proposition~\ref{WEAK}.

For the last term we proceed as above to obtain
\beaa
R_3(\varepsilon, N) &\le& \Pt \Bigl \{ \min \limits_{1 \le k \le \eta(N)} \Psi_k  > 0, \varepsilon \le \frac{\eta(N)}{N^{1-1/\alpha}} \le \varepsilon^{-1}, 1-\varepsilon^2 \le \frac{\Theta_{\eta(N)}}{N} \le 1 \Bigr \} \\
&\le& \sum_{k= \varepsilon N^{1-1/\alpha}}^{\varepsilon^{-1} N^{1-1/\alpha}} \Pt \Bigl \{ \min \limits_{1 \le i \le k} \Psi_i  \ge 0\Bigr \} \Pt \Bigl \{ (1-\varepsilon^2) N \le \Theta_k \le N, \theta_{k+1} > N-\Theta_k \Bigr \}\\
&\le& \frac{\bar{c_1} + o(1)}{\varepsilon^{\frac12} N^{\frac12 - \frac{1}{2\alpha}}} \Pt \Bigl \{ \frac{N - \Theta_{\eta(N)}}{N} \le \varepsilon^2 \Bigr \},
\eeaa
and by Feller~\cite[Ch. XIV.3]{Feller}, the last probability converges to
$$\int_0^{\varepsilon^2} \frac{\sin(\pi (1-1/\alpha)) dx}{\pi x^{1-1/\alpha}(1-x)^{1/\alpha}} < \varepsilon^{\frac{2}{\alpha}}.$$ Combine the estimates above to conclude \eqref{rest} and the proof of the theorem.

\section*{Acknowledgements}
The author is grateful to Vitali Wachtel for discussion and useful references and to Ofer Zeitouni for attention to this work. The author thanks the anonymous referee for his useful comments and suggestions.

\end{document}